\documentclass[11 pt]{amsart}
\usepackage{amscd,amsfonts,amssymb,amsmath}
\usepackage{hyperref}
\usepackage{epsfig}
\usepackage{verbatim}
\usepackage{graphicx}
\newtheorem{theorem}{Theorem}[section]
\newtheorem{corollary}[theorem]{Corollary}

\newtheorem{proposition}[theorem]{Proposition}
\theoremstyle{definition}

\newtheorem{question}[theorem]{Question}
\newtheorem{problem}[theorem]{Problem}

\numberwithin{equation}{subsection}

\usepackage[all,cmtip]{xy}

\newcommand{\Aut}{\operatorname{Aut}}

\newcommand{\id}{\mathrm{id}}

\begin{document}
\title{On the lower central series of some virtual knot groups}
\author{Valeriy G. Bardakov}
\author{Neha Nanda}
\author{Mikhail V. Neshchadim}

\date{\today}
\address{Sobolev Institute of Mathematics and Novosibirsk State University, Novosibirsk 630090, Russia.}
\address{Novosibirsk State Agrarian University, Dobrolyubova street, 160, Novosibirsk, 630039, Russia.}
\address{Regional Scientific and Educational Mathematical Center of Tomsk State University, 36 Lenin Ave., Tomsk, Russia.}
\email{bardakov@math.nsc.ru}

\address{Department of Mathematical Sciences, Indian Institute of Science Education and Research (IISER) Mohali, Sector 81,  S. A. S. Nagar, P. O. Manauli, Punjab 140306, India.}
\email{nehananda@iisermohali.ac.in}

\address{Sobolev Institute of Mathematics and Novosibirsk State University, Novosibirsk 630090, Russia.}
\address{Regional Scientific and Educational Mathematical Center of Tomsk State University, 36 Lenin Ave., Tomsk, Russia.}
\email{neshch@math.nsc.ru}

\subjclass[2010]{Primary 57M27; Secondary 20F36, 57M25}
\keywords{Knot group, representation, virtual braid group, virtual knot, residually nilpotent group}

\begin{abstract}
We study groups of some virtual knots with small number of crossings and prove that there is a virtual knot with long lower central series which, in particular, implies that there is a virtual knot with residually nilpotent group. This gives a possibility to construct invariants of virtual knots using quotients by terms of the lower central series of knot groups. Also, we study decomposition of virtual knot groups as semi direct product and free product with amalgamation. In particular, we prove that the groups of some virtual knots are extensions of finitely generated free groups by infinite cyclic groups.
\end{abstract}
\maketitle

\section{Introduction}\label{Section1}
Virtual links were introduced by  Kauffman \cite{Kauffman2} as a generalization of classical links. Topologically, virtual links can be interpreted as isotopy classes of embeddings of classical links in thickened surfaces of higher genus. Several invariants of classical links can be extended to virtual links. For example, Kauffman in \cite{Kauffman2}, defined virtual link groups using Wirtinger's algorithm by ignoring virtual crossings.

Various definitions of virtual link groups can be found in the literature \cite{Bardakov1, Bardakov2, Bardakov3, Bardakov4, BDG, CSW, Mikh, Silver}, where some of these definitions use representations of  the virtual braid group $VB_n$ by automorphisms of  free products $F_{n,k}=F_n \ast \mathbb{Z}^k$ for some $k$.\par

It is well-known that the Artin braid group $B_n$ can be represented as a subgroup of $\Aut (F_n)$. In \cite{Bardakov1, Manturov}, an extension of the classical Artin representation $\varphi_A : VB_n \to \Aut(F_{n+1})$ is constructed. Using the representation $\varphi_A$, for each virtual link $L$, a group $G_A(L)$ is defined. It is shown in \cite{Bardakov1} that the group $G_A(L)$ distinguishes virtual trefoil knot and unknot, unlike the group defined in \cite{Kauffman2}. Therefore, it is interesting to study the virtual knot groups defined using representations of $VB_n$. Later, in \cite{BDG}, a more general representation of $VB_n$ and its associated virtual link group is defined. Recently, in \cite{Bardakov3}, a new representation $\tilde{\varphi}_{M}$ of the virtual braid group is constructed, which generalizes all previously mentioned representations. If $K$ is a virtual knot and $G_{\tilde{M}}  (K)$ the group associated to representation $\tilde{\varphi}_{M}$, then it has been shown that $G_A(K)$ is isomorphic to $G_{\tilde{M}}  (K)$. A natural question is whether the same holds for virtual links with more than one component. In this paper, we show that if $L$ is the virtual Hopf link, then $G_A(L)$ is not isomorphic to $G_{\tilde{M}}(L)$.

 It is well-known that the knot group for classical knots has a short lower central series, that is, the second term coincides with the third term \cite[Section 6, Page 59]{Neuwirth}. Therefore, factorization by the terms of lower central series cannot be used to distinguish classical knots. On the other hand, there are classical links whose groups are residually nilpotent \cite{BM}. 
 \par
In \cite{Bardakov5}, it is shown that the group $G_A(K)$ of virtual trefoil knot $K$ has a homomorphism onto a nilpotent group of step 4. Also it was asked whether there is a  virtual knot whose knot group is residually nilpotent.  In this paper, we give a positive answer to this question. Also, we show  that there is a virtual knot whose group has lower central series of length $\leq \omega^2$, where $\omega$ is the first infinite ordinal.\par
The paper is organised as follows. In section \ref{Section2}, we recall definition of the virtual braid group $VB_n$ and definitions of the virtual link groups $G_A(L)$ and $G_{\tilde{M}}(L)$. Also, we proved that the second and the third term of virtual knot group defined by Kauffman coincides.
\par
In section \ref{Section2.1}, we study cyclic extensions of free groups. In section \ref{section3}, we consider virtual knots with small number of  crossings from  the virtual knot table \cite{table}.  We compute their knot groups corresponding to representation $\varphi_A$. Further, we study quotients of virtual knot groups by terms of lower central series, and prove that all these knots are nontrivial using the quotient by the fifth term of the lower central series. We show that there are virtual knots whose knot groups have long lower central series. We also study structures of these groups and prove that some of them  are extensions of finitely generated free groups by infinite cyclic groups.
\par
In section \ref{section4},
we consider the virtual Hopf link $L$ and prove that $G_A(L)$ and $G_{\tilde{M}}(L)$ are right angled Artin groups but  $G_A(L)$ is not isomorphic to $G_{\tilde{M}}(L)$.
\bigskip

\section{Preliminaries and Notations}\label{Section2}
Throughout the paper, we shall denote $x^y=y^{-1}xy$, $[x, y] = x^{-1} y^{-1} x y$, where $x, y$ are elements of some group. If $A$ and $B$ are subgroups of a group $G$, then $[A, B]$ is a subgroup that is generated by all commutators $[a, b]$, $a \in A$, $b \in B$. In particular, we denote $G'= [G, G]$ and  $G''= [G', G']$.
A nontrivial group $G$ is called {\it residually nilpotent} if for any $1 \not= g \in G$ there is a nilpotent group $N$ and a homomorphism $\varphi : G \longrightarrow N$ such that $\varphi(g) \not= 1$.\\
\\
For a group $G$, transfinite lower central series is defined as,
$$
G = \gamma_1 (G) \geq \gamma_2 (G) \geq \ldots \geq \gamma_{\omega} (G) \geq \gamma_{\omega+1} (G) \geq \ldots,
$$
where
$$
\gamma_{\alpha+1} (G) = \langle [g_{\alpha}, g] ~|~g_{\alpha} \in \gamma_{\alpha} (G), g \in G \rangle,
$$
and if $\alpha$ is a limit ordinal, then
$$
\gamma_{\alpha} (G) = \bigcap_{\beta < \alpha}\gamma_{\beta} (G).
$$
In particular, $G$ is residually nilpotent if and only if
$$
\gamma_{\omega}(G) = \bigcap_{i=1}^{\infty} \gamma_{i}(G) = 1.
$$
The maximal $\alpha$ such that $\gamma_{\alpha} (G) \not = \gamma_{\alpha+1} (G)$ is called the {\it length of the lower central series} of $G$.\\
\\
As we noted in the introduction,  if $K$ is a  classical knot, then its group $G(K)$ is not residually nilpotent since $\gamma_2(G(K)) = \gamma_3(G(K))$.

\medskip

The {\it virtual braid group} $VB_n$ is generated by the classical braid group $B_n = \langle \sigma_1, \sigma_2, \ldots, \sigma_{n-1} \rangle$ and the symmetric group $S_n = \langle \rho_1, \rho_2,\ldots, \rho_{n-1} \rangle$ such that the following relations hold:
\begin{align*}
    \sigma_i \sigma_{i+1} \sigma_i&=\sigma_{i+1} \sigma_i \sigma_{i+1} & i=1, 2, \ldots, {n-2}, \\
\sigma_i \sigma_j&=\sigma_j \sigma_i &  |i-j| \geq 2, \\
\rho_i^{2}&=1 &  i=1, 2, \ldots, {n-1},\\
\rho_i \rho_j&= \rho_j \rho_i &  |i-j| \geq 2,\\
\rho_i \rho_{i+1} \rho_i&= \rho_i \rho_{i+1} \rho_i & i=1, 2, \ldots, {n-2}.
\end{align*}
In addition, the following mixed defining relations hold:
\begin{align*}
    \sigma_i \rho_j&= \rho_j \sigma_i & |i-j| \geq 2 ,\\
    \rho_i \rho_{i+1} \sigma_i&= \sigma_{i+1} \rho_i \rho_{i+1} & i=1, 2, \ldots, {n-2}.
\end{align*}
The generators $\sigma_i$ and $\rho_i$ can be geometrically presented as shown in the figure below.

\begin{figure*}[hbtp]
\centering
\includegraphics[scale=0.7]{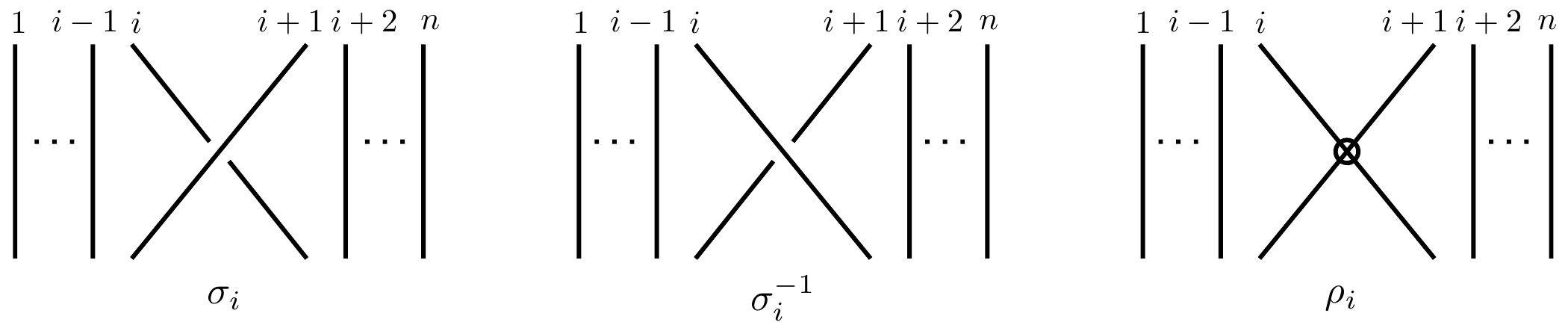}
\end{figure*}

Recall the two representations of the virtual braid group, which are extensions of the Artin representation  of $B_n$ into $ \Aut(F_n)$ \cite{Artin}.
Let $F_{n+1}= \langle x_1, x_2,\ldots,x_n, y \rangle$ be the free group of rank $n+1$. The map $$\varphi_A : VB_n \to \Aut(F_{n+1})$$ defined by setting

$$
\varphi_{A}(\sigma_i) = \left\{
        \begin{array}{ll}
      x_i \mapsto x_ix_{i+1}x_i^{-1} ,  \\
        x_{i+1} \mapsto x_i,  \\
        x_j \mapsto x_j~\textrm{for}~~j \neq i, i+1, \\
             y \mapsto y.  \\
        \end{array}
    \right.
$$
and
$$
\varphi_{A}(\rho_i) = \left\{
        \begin{array}{ll}
      x_i \mapsto x_{i+1}^{y^{-1}} ,  \\
    x_{i+1} \mapsto x_{i}^y,  \\
                    x_j \mapsto x_j~\textrm{for}~~j \neq i, i+1, \\
             y \mapsto y. \\
        \end{array}
    \right.
$$
is a representation of $VB_n$  to $\Aut(F_{n+1})$.
\bigskip

Next, we consider the second representation of $VB_n$. Let $F_{n,n} =  F_n * {\mathbb{Z}}^{n} $, where $F_{n}= \langle y_1, y_2,\dots,y_n \rangle$ and  $\mathbb{Z} ^n= \langle v_1, v_2,\dots,v_n \rangle$ is the free abelian group of rank $n$. The map
$$ \tilde \varphi_{M} : VB_n \to \Aut(F_{n,n}) $$
defined by setting
$$
\tilde \varphi_{M}(\sigma_i) = \left\{
        \begin{array}{ll}
      y_i \mapsto y_iy_{i+1}y_i^{-1} ,  \\
    y_{i+1} \mapsto y_i,\\
            y_j \mapsto y_j~\textrm{for}~~j \neq i, i+1. \\
        \end{array}
    \right.~~~
\tilde \varphi_{M}(\sigma_i) = \left\{
        \begin{array}{ll}
      v_i \mapsto v_{i+1} ,  \\
    v_{i+1} \mapsto v_i,  \\
                        v_j \mapsto v_j~\textrm{for}~~j \neq i, i+1. \\
        \end{array}
    \right.
$$
$$
\tilde \varphi_{M}(\rho_i) = \left\{
        \begin{array}{ll}
      y_i \mapsto y_{i+1}^{v_{i}^{-1}} ,  \\
    y_{i+1} \mapsto y_{i}^{v_{i+1}},  \\
                        y_j \mapsto y_j~\textrm{for}~~j \neq i, i+1. \\
        \end{array}
    \right.~~~
\tilde \varphi_{M}(\rho_i) = \left\{
        \begin{array}{ll}
      v_i \mapsto v_{i+1} ,  \\
    v_{i+1} \mapsto v_i,  \\
                                    v_j \mapsto v_j~\textrm{for}~~j \neq i,i+1.\\
        \end{array}
    \right.
$$
is  a representation of $VB_n$ \cite{Bardakov3}.\\

In classical knot theory, Alexander theorem states that for each oriented link $L$, there exists a braid whose closure is equivalent to $L$ \cite[Chapter 2, Theorem 2.3]{Kassel}. The closure of a virtual braid is same as the closure defined for classical braids which is shown in Figure \ref{CLOSURE}.  An analogous result holds for virtual links and virtual braids \cite{Kauffman3}. Similarly,  an analogue of Markov theorem for virtual braids has been established by Kamada \cite{Kamada}. Let $L$ be a virtual link which is equivalent to closure of a virtual braid $\beta \in VB_n$. A braid representative for any virtual knot can be obtained by the braiding process defined in \cite[Section 4]{Kamada}. Suppose that we have a representation $\varphi : VB_n \to \Aut(H)$ of the virtual braid group into the automorphism group of a group $H = \langle h_1, h_2,\dots, h_m ~|~ \mathcal{R} \rangle$, where $\mathcal{R}$ is the set of defining relations. Then to each $\beta \in VB_n$,  we assign the group
\begin{equation*}\label{group-rep}
    G_{\varphi}(\beta) = \langle h_1,  h_2, \dots, h_m ~\|~ \mathcal{R}, h_i= \varphi(\beta)(h_i)~\textrm{for}~ i= 1,2,\dots,m \rangle.
\end{equation*}
\begin{figure}[hbtp]
\centering
\includegraphics[scale=1]{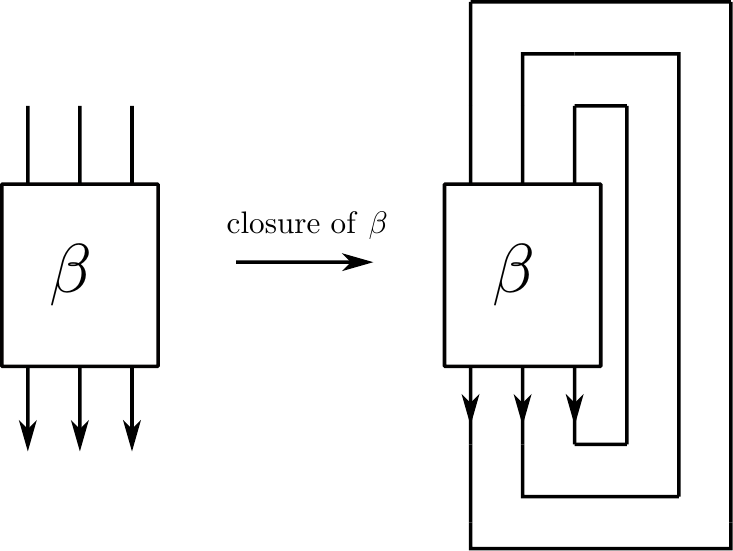}
\caption{Closure of a virtual braid $\beta$.}
\label{CLOSURE}
\end{figure}
Using the representation $\varphi_A$, we define a group $G_{\varphi_A}(\beta)$, which for simplicity, we shall denote by  $G_{A}(\beta)$. It is shown in \cite{Bardakov2} that this group is a link invariant, by showing that if two virtual braids are related by Markov moves as defined in \cite{Kamada}, then their groups are isomorphic. Later, it is shown in  \cite{Bardakov3} that the group $G_{\widetilde{\varphi}_{M}}(\beta)$, corresponding to the representation $\tilde{\varphi}_M$, is a link invariant. It is not difficult to prove that if $K$ is a trivial knot, then its group is isomorphic to $F_2$, the free
group of rank two. Throughout the paper, we denote $G_A(\beta)$ by $G_A(K)$.
\par
In \cite[Section 6]{Bardakov3}, it is shown that for any virtual knot $K$, its group $G_A(K)$ can also be computed using its diagram. The diagram is divided into arcs from one crossing (real or virtual) to the other and each arc is assigned a symbol. To each real and virtual crossing, we assign the relations as shown in Figure \ref{KnotDiagramRelations}. Then we consider the group generated by symbols assigned on arcs and relations obtained from the crossings. The group obtained is isomorphic to the group $G_A(K)$.\\

\begin{figure}[hbtp]
\centering
\includegraphics[scale=1]{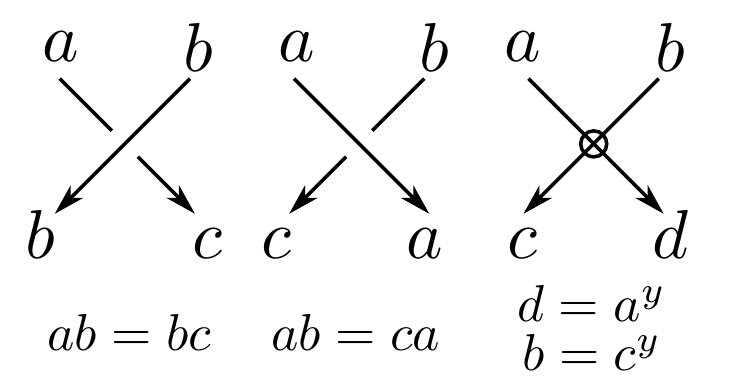}
\caption{Relations from crossings of a virtual knot diagram}
\label{KnotDiagramRelations}
\end{figure}

For a virtual knot $K$, note that the quotient of the group $G_A(K)$ by relation $y = 1$ is the Kauffman group $G_{Ka}(K)$ defined in \cite{Kauffman2}. For the trivial knot this group is infinite cyclic. We prove the following result analogous to the classical knots.

\begin{proposition}\label{theorem}
For any virtual knot $K$, we have $\gamma_2 (G_{Ka}(K)) = \gamma_3 (G_{Ka}(K))$.
\end{proposition}

\begin{proof}
In \cite[Theorem 2.2]{Williams}, a characterization for virtual knot groups is given. It is shown that $G_{Ka}(K)$ has presentation 
$$\langle x_1, x_2, \dots, x_n ~|~ x_{i+1}= u_i^{-1}x_iu_i, ~ i \in \mathbb{Z}_n\rangle, $$ where $u_i$ are elements in free group generated by $x_1, x_2, \dots, x_n$. Therefore, the abelianisation of $G$ is $\mathbb{Z}$ and the rest of the proof follows from \cite[Section 6, Page 59]{Neuwirth}.
\end{proof}

 \section{Cyclic extensions of free groups} \label{Section2.1}

We begin by denoting $[b, {}_k a]$ as the left normalized commutator  $[b, a, \ldots, a]$, where $a$ is repeated $k$ times. For example,
$$
[b, {}_1 a] = [b,  a],~~~[b, {}_k a] = [[b, {}_{k-1} a], a].
$$
In this section, we prove the following general theorem.

\begin{theorem} \label{t}
Let $G$ be an extension of the free group $F$ of finite rank $n\geq 2$ by an infinite cyclic group $\left\langle a \right\rangle $ that is, $F\unlhd G$ and $G/F=\mathbb{Z}=\left\langle a \right\rangle $. On the quotient $A = F / \gamma_2 F$, conjugation by  element $a$ induces a linear  map $\alpha$ which acts by the rule:
$$ 
\alpha(x)  \gamma_2 F = x^a  \gamma_2 F   ,~~ x \in F. 
$$
\begin{enumerate}
\item If $(\alpha - \id)^n=0$, where $\id$ is the identity map, then  $G$ is residually nilpotent that is,  $\gamma_{\omega} G=1$. 

\item If  $(\alpha- \id)^m A \subseteq M A$ for some $m, M \in \mathbb{N}$ and $M \geq 2$, then the length of the  lower central series for $G$ is less than or equal to $\omega^2$, in particular, $\gamma_{\omega^2} G=1$. 
 \end{enumerate}
\end{theorem}

\begin{proof}
\begin{enumerate}
\item Using induction on $k \in \mathbb{N}$, let us show that the following inclusion
$$
\gamma_{n \frac{n^k-1}{n-1}} G \subseteq \gamma_{k+1} F
$$
holds. For $k=1$, we need to show that $\gamma_nG \subseteq \gamma_2{F}.$ For any $x \in F$, we have 
\begin{align*}
[x, a] &= x^{-1}a^{-1}xa = x^{-1}x^a \equiv \alpha(x)-x\\
&= (\alpha - \id)(x)(\mathrm{mod}\gamma_2{F}).
\end{align*}
Therefore, 
$$[x ,_n a] \equiv  (\alpha - \id)^n(x)= 0(\mathrm{mod} \gamma_2{F}).$$
  Let us suppose that our hypothesis holds for  $k-1$.
We consider the quotient $\gamma_{k} F /\gamma_{k+1} F$ which is the homomorphic image of the tensor product
$$
{A}^{\otimes k}={A} \otimes \ldots \otimes {A}.$$ 
Indeed, the  quotient $\gamma_{k} F /\gamma_{k+1} F$ is generated by the left normalized commutators  $[y_1,\dots, y_k]$, $y_i \in F$, which are multilinear functions of their arguments.
Hence, we have the natural surjection
$$
x_1\otimes \ldots \otimes x_k \mapsto [x_1\otimes \ldots \otimes x_k]
$$
from the tensor product  ${A}^{\otimes k}$ to
$\gamma_{k} F /\gamma_{k+1} F$, where $x_1, \ldots, x_k\in A$.

We note that the action of $G$ by conjugation on the groups $\gamma_{k} F /\gamma_{k+1} F$ and
${A}^{\otimes k}$ is compatible with homomorphism
${A}^{\otimes k} \longrightarrow \gamma_{k} F /\gamma_{k+1} F$. Therefore, it suffices to prove that
 $$
[{A}^{\otimes k}, {}_{n^k}a]=0.
$$

For any element $x_1\otimes \ldots \otimes x_k$ in ${A}^{\otimes k}$, we have
\begin{align*}
[x_1\otimes \ldots \otimes x_k,a]&=x_1^a\otimes \ldots \otimes x_k^a -x_1\otimes \ldots \otimes x_k\\
&=\alpha x_1\otimes \ldots \otimes \alpha x_k -x_1\otimes \ldots \otimes x_k\\
&=(\alpha- \id +\id) x_1\otimes \ldots \otimes (\alpha- \id +\id) x_k -x_1\otimes \ldots \otimes x_k\\
&=\left((\alpha-\id)\otimes \ldots \otimes (\alpha-\id)+ \ldots+  \id\otimes \ldots \otimes (\alpha-\id) \right)
 x_1\otimes \ldots \otimes x_k.
\end{align*}
Therefore, we have the following inclusion
$$
[{A}^{\otimes k}, a]\subseteq T_k(\alpha) {A}^{\otimes k},
$$
where
$$
T_k(\alpha) = (\alpha-\id)\otimes \id \otimes \ldots \otimes \id+ \ldots+  \id\otimes \ldots \otimes \id\otimes (\alpha - \id).
$$
Note that the operator $\alpha - \id$  is included only once in each term of $T_k(\alpha)$. Further, it is not difficult to show that
$$
[{A}^{\otimes k}, {}_{m}a]\subseteq T_k^m(\alpha){A}^{\otimes k}.
$$
Now if $m\geq n^k$, then using the fact that  $\alpha-\id$ is nilpotent, we get $T_k^m(\alpha)=0$, and hence
 $$
[\gamma_k F, {}_{n^k}a]\subseteq \gamma_{k+1} F.
$$

\item Let us suppose $(\alpha- \id)^m A \subseteq M A$, for some
$m, M \in \mathbb{N}$ and $M \geq 2$. Using induction on $k\in \mathbb{N}$, let us show that the following inclusion holds
$$
\gamma_{(k+1)\omega} G \subseteq \gamma_{k} F.
$$
Since $G/F = \mathbb{Z}$, we have $\gamma_\omega G \subseteq F$.
Suppose that our hypothesis holds for $k-1$. We know that $\gamma_{k} F /\gamma_{k+1} F$ is the homomorphic image of tensor product
${A}^{\otimes k}$, therefore it is enough to show that
$$
 \bigcap\limits_{l\geq 1} [{A}^{\otimes k}, {}_{l}a]=0.
$$
Similar to the proof of (1), we get
$$
[{A}^{\otimes k}, {}_{m}a]\subseteq T_k^m(\alpha){A}^{\otimes k}.
$$
For $m\geq n^k$, we have
$$
T_k^m(\alpha){A}^{\otimes k} \subseteq M {A}^{\otimes k}.$$
Therefore,
$$
 \bigcap\limits_{l\geq 1} [{A}^{\otimes k}, {}_{l}a]\subseteq
 \bigcap\limits_{s\geq 1} M^s {A}^{\otimes k}=0.$$
\end{enumerate}
\end{proof}

\section{Computations for virtual knots with small number of crossings}
\label{section3}
In this section, we study some virtual knots with small number of crossings from the knot table \cite{table}. Using braid representatives, we compute $G_A(K)$ for virtual knots $K_1$ and $K_3$, whereas for $K_2$ and $K_4$, we use the virtual knot diagram approach. We then study structures of these groups.

\subsection{Virtual knot  $K_1$ (The knot $2.1$ in \cite{table})}
It is the virtual trefoil knot and is the closure of the braid $\beta=\sigma_1^{-2} \rho_1  \in VB_2$ as shown in Figure \ref{BraidForK_1}.
Its group has the following presentation
$$
G(K_1) = \langle x_1, x_2, y ~||~x_1 = y^{-1} x_1^{-1} y^2 x_2 y^{-2} x_1 y, ~~ 
x_2 = y^{-1} x_1^{-1} y^2 x_2^{-1} y^{-2} x_1 y \cdot y x_2  y^{-2} x_1 y \rangle.
$$
Using the first relation, we can rewrite the second relation in the form
$$
x_2 =  x_1^{-1} \cdot y x_2  y^{-2} x_1 y,
$$
or
$$
x_1 x_2 =   y x_2  y^{-2} x_1 y.
$$
Using this relation, we rewrite the first one, and we have 
$$
G(K_1) = \langle x_1, x_2, y ~||~x_1 = y^{-1} x_1^{-1} y x_1 x_2, ~~ x_2 =  x_1^{-1} y x_2  y^{-2} x_1 y  \rangle.
$$
From the first relation, it follows that $x_2 = x_1^{-1} y^{-1} x_1 y x_1$. So we remove the generator $x_2$ to get
$$
G(K_1) = \langle x, y ~||~x^{-1}  y^{-1} x^{-1} y^2 x^{-1} y^{-1}  x y x   y^{-2} x y = 1 \rangle = \langle x, y ~||~x =  x^{ y x   y^{-2} x y } \rangle,
$$
where $x = x_1$. 
It is easy to check that
$$
G(K_1) = \langle\, x,\,y\, \| \, [x^{-1},y,x^{-1},yx^{-1}]=1 \,  \rangle.
$$
\begin{figure*}[htbp]
\centering
\includegraphics[scale=0.3]{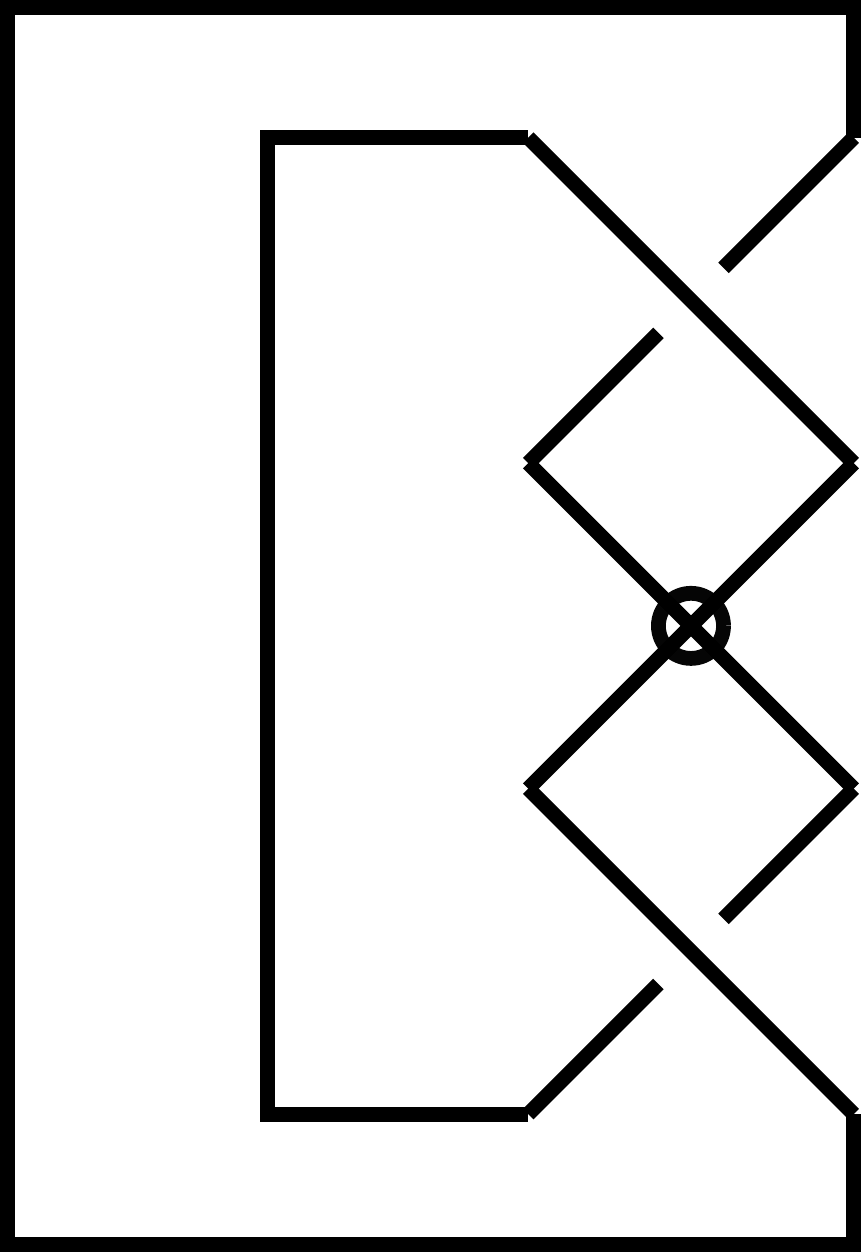}
\caption{Knot $K_1$.}
\label{BraidForK_1}
\end{figure*}
\begin{theorem}
\enumerate
\item For the knot $K_1$, we have
\begin{eqnarray*}
G(K_1)/\gamma_4 G(K_1) & \cong & F_2/\gamma_4 F_2,\\
\gamma_4 G(K_1)/\gamma_5 G(K_1) & \cong &  \mathbb{Z}^2,\\
G(K_1)/\gamma_5 G(K_1) & \not\cong &  F_2/\gamma_5 F_2.
\end{eqnarray*}
\item The group $G(K_1)$ is an extension of free group of rank  3 by infinite cyclic group.
\end{theorem}

\begin{proof} \enumerate
\item We see that the left hand side of the defining relation of $G(K_1)$ lies in $\gamma_4 F_2$, where $F_2 = \langle x, y \rangle$. Hence, $G(K_1)/\gamma_4 G(K_1) \cong  F_2/\gamma_4 F_2$.
    
The quotient $\gamma_4 G(K_1) / \gamma_5 G(K_1)$ is generated by the basis commutators of weight 4.
In the quotient $G(K_1)/\gamma_5 G(K_1)$, we have
$$
1=[x^{-1},y,x^{-1},yx^{-1}]\equiv[x^{-1},y,x^{-1},y] [x^{-1},y,x^{-1},x^{-1}] \equiv[x,y,x,y] [x,y,x,x]^{-1}.
$$
Therefore, the defining relation in $G(K_1)/\gamma_5 G(K_1)$ is of the form 
$$[x,y,x,y] \equiv [x,y,y,x].$$

The commutators  $[x,y,y,x]$, $[x,y,x,x]$, $[x,y,y,y]$
are the basic commutators of weight 4, i.e. $\gamma_4 F_2 / \gamma_5 F_2 \cong \mathbb{Z}^3$, whereas from the defining relation of $G(K_1)$, it follows that 
\begin{eqnarray*}
\gamma_4 G(K_1)/\gamma_5 G(K_1) & \cong &  \mathbb{Z}^2.
\end{eqnarray*}
Hence, we obtain
\begin{eqnarray*}
G(K_1)/\gamma_5 G(K_1) & \not\cong &  F_2/\gamma_5 F_2.
\end{eqnarray*}
\bigskip

\item Rewriting the relation of $G(K_1)$ , we have
$$
x^{-1}y^{-1}xy= y^{-2} xyxy^{-1}x^{-1}y^2x^{-1}.
$$
Putting
$$
y_k=x^{-k}yx^{k},\quad k \in \mathbb{Z},
$$
the relation is equivalent to 
$$
y_1^{-1}y_0= y_0^{-2} y_{-1}y_{-2}^{-1}y_{-1}^2.
$$
On shifting the indexes by two, we get
$$
y_3^{-1}y_2= y_2^{-2} y_{1}y_{0}^{-1}y_{1}^2,
$$
which implies
$$
y_3= y_2 y_{1}^{-2}y_{0}y_{1}^{-1}y_2^2.
$$
Therefore,  $G(K_1)$ has the presentation
$$
G(K_1)=\left\langle \, x,\,y_0,\,y_1,\,y_2\, \| \,
y_0^x=y_1,\, y_1^x=y_2,\, y_2^x=y_2 y_{1}^{-2}y_{0}y_{1}^{-1}y_2^2,\,
     \right.
$$
$$
  \left.
y_2^{x^{-1}}=y_1,\, y_1^{x^{-1}}=y_0,\, y_0^{x^{-1}}=y_0^2 y_{1}^{-2}y_{2}y_{1}^{-2}y_0 \, \right\rangle.
$$
So we have, 
$$
G(K_1)=F_3 \rtimes \mathbb{Z},
$$
where $F_3=\langle y_0,\,y_1,\,y_2 \rangle$ is the free group of rank  3 and
$\mathbb{Z}=\langle x \rangle$.

\end{proof}
Conjugation of $F_3$ by $x$ induces a linear transformation on the quotient
  $F_3/ F_3'$, which has  the matrix
$$
A=
\left(%
\begin{array}{ccc}
  0 &  1 & 0 \\
  0 &  0 & 1 \\
  1 & -3 & 3 \\
\end{array}%
\right).
$$
Since all the eigenvalues of $A$ are equal to 1, we have  $(A-E)^3=0$.
Using the first part of Theorem \ref{t},
we have the following result.

\begin{corollary}
The group  $G(K_1)$ is residually nilpotent.
\end{corollary}
\medskip
\subsection{Virtual knot $K_2$ (The knot $4.21$ in \cite{table}) } 

\begin{figure}[t]
\centering
\includegraphics[scale=0.7]{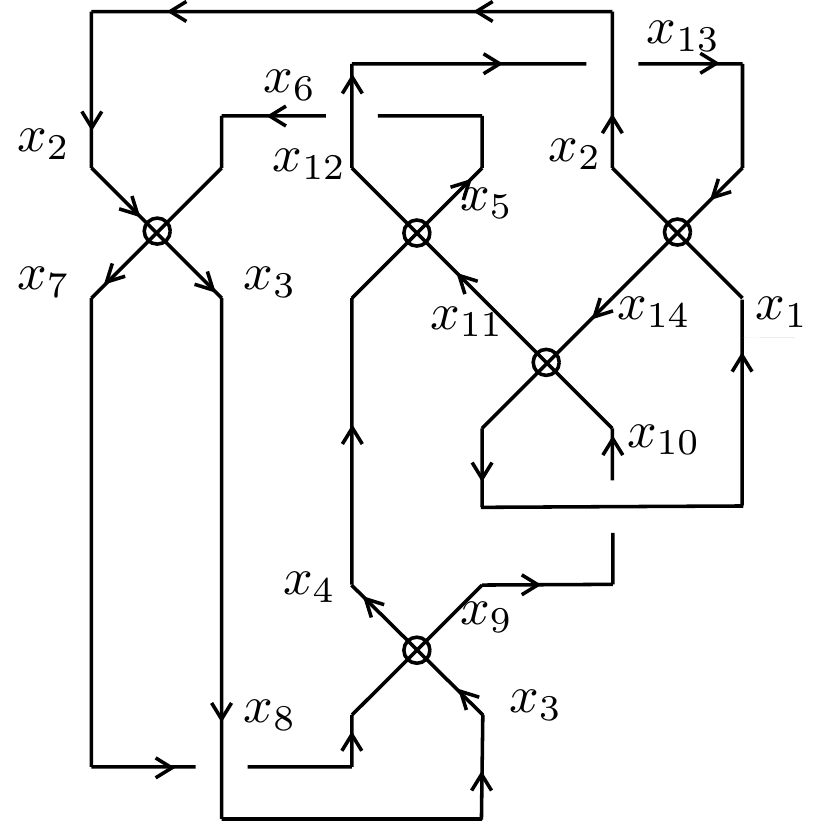}
\caption{Knot $K_2$.}
\label{BraidForK_2}
\end{figure}

We compute the group $G(K_2)$ using its knot diagram. By labelling the arcs of $K_2$ as in Figure \ref{BraidForK_2}, we get
$$G(K_2)= \langle x_1, x_2, \dots , x_{14}, y ~|~ x_5x_{12}=x_{12}x_6, x_2x_{12}=x_{13}x_2, x_7x_3=x_3x_8, x_9x_1=x_1x_{10}, x_3=x_2^y,$$
$$ x_6=x_7^y, x_{12}=x_{11}^y, x_4=x_5^y, x_{14}=x_{13}^y, x_1=x_2^y, x_4=x_3^y, x_8=x_9^y, x_1=x_{14}^y, x_{10}=x_{11}^y \rangle .$$

Simplifying the above presentation and substituting $x_1=x$, we get 

$$G(K_2) = \langle x, y ~|~ xx^{-y^{-1}}x^{y^{-2}}x^{y^{-1}}=x^{-y^{-1}}x^{y^{-2}}x^{y^{-1}}(x(x_1x_1^{-y^{-1}}x_1^{y^{-2}}x_1^{y^{-1}}x_1^{-1})^yx^{-1})^y \rangle. $$
Simplifying the above relation, we obtain 
$$ yx^{-1}yxy^{-1}xy^{-2}xy^{-1}xyx^{-1}yx=xyx^{-1}yxy^{-1}xy^{-2}xy^{-1}xyx^{-1}y.$$
Therefore, we have 
$$G(K_2) = \langle\, x,\,y\, \| \, [x^{y^{-1}xy^{-1}}  x^{yx^{-1}y},x]=1 \rangle.$$

Simplifying the relation and using the commutator identities, we get
\begin{eqnarray*}
\lbrack x [x,y^{-1}xy^{-1} \rbrack \,  x[x,yx^{-1}y],x \rbrack &=& 1,\\
\lbrack x^2 [x,y^{-1}xy^{-1}]\, [x,y^{-1}xy^{-1},x]\, [x,yx^{-1}y],x \rbrack &=& 1,\\
\lbrack  [x,y^{-1}xy^{-1}]\, [x,y^{-1}xy^{-1},x]\, [x,yx^{-1}y],x \rbrack &=& 1.
\end{eqnarray*}
Performing the  transformations modulo  $\gamma_5 G(K_2)$, we obtain
\begin{eqnarray*}
 \lbrack x,y^{-1}xy^{-1},x\rbrack \, [x,y^{-1}xy^{-1},x,x]\, [x,yx^{-1}y,x] & \equiv & 1,\\
\lbrack x,y^{-1}xy^{-1},x \rbrack \, [x,y,x,x]^{-2}\, [x,yx^{-1}y,x] & \equiv & 1,\\
\lbrack [y^{-1}xy^{-1},x]^{-1},x\rbrack \, [[yx^{-1}y,x]^{-1},x] & \equiv &  [x,y,x,x]^{2},\\
\lbrack x,[y^{-1}xy^{-1},x]\rbrack \, [x,[yx^{-1}y,x]] & \equiv &  [x,y,x,x]^{2},\\
\lbrack y^{-1}xy^{-1},x,x\rbrack ^{-1}\, [yx^{-1}y,x,x]^{-1} & \equiv & [x,y,x,x]^{2},\\
 \lbrack x^{-1}y^{-1}xy^{-1},x,x\rbrack ^{-1}\, [xyx^{-1}y,x,x]^{-1} & \equiv &  [x,y,x,x]^{2},\\
\lbrack [x,y]y^{-2},x,x\rbrack ^{-1}\, [[x^{-1},y^{-1}]y^2,x,x]^{-1}  & \equiv & [x,y,x,x]^{2},\\
\lbrack [x,y,x]\,[y^{-2},x],x\rbrack ^{-1}\, [[x^{-1},y^{-1},x]\,[y^2,x],x]^{-1} &\equiv & [x,y,x,x]^{2},\\
 \lbrack  x,y,x,x\rbrack ^{-1}\,[y^{-2},x,x]^{-1}\, [x^{-1},y^{-1},x,x]^{-1}\,[y^2,x,x]^{-1} & \equiv& [x,y,x,x]^{2},\\
\lbrack  y^{-2},x,x\rbrack ^{-1}\, [y^2,x,x]^{-1} & \equiv & [x,y,x,x]^{4}.\\
\end{eqnarray*}

Further, we have
$$
 [y^{-2},x,x]=[[x,y^{2}]^{y^{-2}},x] \equiv [[x,y^{2}]\,[x,y^{2},y^{-2}] ,x] \equiv [x,y^{2},x]\,[x,y,y,x]^{-4} \equiv
$$
$$
 \equiv[[y^{2},x]^{-1},x]\,[x,y,y,x]^{-4} \equiv [x,[y^{2},x]]\,[x,y,y,x]^{-4} \equiv [y^{2},x,x]^{-1}\,[x,y,y,x]^{-4},
$$
Therefore, in the quotient  $G(K_2) / \gamma_5 G(K_2)$ the following unique relation holds
$$
 ([x,y,y,x]^{-1}[x,y,x,x])^{4}\equiv 1.
$$
The commutators  $[x,y,y,x]$, $[x,y,x,x]$, $[x,y,y,y]$
are the basic commutators of weight  4, hence we have proved the first part of the following theorem:

\begin{theorem}
\begin{enumerate}
\item The first five terms of lower central series of  the group $G(K_2)$ are different from each other. Moreover, we have
\begin{eqnarray*}
G(K_2)/\gamma_4 G(K_2) & \cong &  F_2/\gamma_4 F_2,\\
\gamma_4 G(K_2)/\gamma_5 G(K_2) & \cong &  \mathbb{Z}^2 \times \mathbb{Z}_4,\\
G(K_2) / \gamma_5 G(K_2) & \not\cong &  F_2/\gamma_5 F_2.
\end{eqnarray*}
\item  The group $G(K_2)$ is an extension of free group of rank  5
by infinite cyclic group.
\end{enumerate}
\end{theorem}
\begin{proof}
(2) Consider the relation of $G(K_2)$:
$$
[x^{y^{-1}xy^{-1}}  x^{yx^{-1}y},x]=1
$$
Transforming it, we obtain
\begin{eqnarray*}
x^{y^{-1}xy^{-1}}\,  x^{yx^{-1}y}\, x &=& x\, x^{y^{-1}xy^{-1}}\,  x^{yx^{-1}y},\\
x^{-3}x^{y^{-1}xy^{-1}}\,  x^{yx^{-1}y}\, x  &=& x^{-2}\, x^{y^{-1}xy^{-1}}\,  x^{yx^{-1}y},\\
x^{-3} ( yx^{-1}y  x y^{-1}xy^{-1})  \, (y^{-1}xy^{-1}  x yx^{-1}y) \, x  &=& x^{-2} ( yx^{-1}y  x y^{-1}xy^{-1})  \, (y^{-1}xy^{-1}  x yx^{-1}y).
\end{eqnarray*}

Putting
$$
y_k=x^{-k}yx^{k},\quad k \in \mathbb{Z}.
$$
The above relation has the form
$$
y_3 y_4 y_3^{-1} y_2^{-2}y_1^{-1}y_0 y_1=y_2 y_3 y_2^{-1} y_1^{-2}y_0^{-1}y_{-1} y_0.
$$
Shifting the indexes by 1, we get
\[
y_4 y_5 y_4^{-1} y_3^{-2}y_2^{-1}y_1 y_2=y_3 y_4 y_3^{-1} y_2^{-2}y_1^{-1}y_0 y_1.
\]
Therefore, we have
$$
y_5  = y_4^{-1} y_3 y_4 y_3^{-1} y_2^{-2}y_1^{-1}y_0 y_1y_2^{-1}y_1^{-1}y_2y_3^{2}y_4,
$$
Hence, $G(K_2)$ in the generators  $x$, $y_0$, $y_1$, $y_2$, $y_3$, $y_4$ has the following presentation:
$$
G(K_2)= \langle x,\,y_0,\,y_1,\,y_2,\,y_3,\,y_4\, \| \, y_0^x=y_1,\, y_1^x=y_2,\,y_2^x=y_3,\, y_3^x=y_4,\,
$$
$$
y_4^x=y_4^{-1} y_3 y_4 y_3^{-1} y_2^{-2}y_1^{-1}y_0 y_1y_2^{-1}y_1^{-1}y_2y_3^{2}y_4,\,
y_4^{x^{-1}}=y_3,\, y_3^{x^{-1}}=y_2,\,y_2^{x^{-1}}=y_1,\, y_1^{x^{-1}}=y_0,\,
$$
$$
y_0^{x^{-1}}=y_0y_1^{2}y_2 y_3^{-1}y_2^{-1} y_3 y_4 y_3^{-1} y_2^{-2}y_1^{-1}y_0 y_1y_0^{-1}
 \, \rangle.
$$
Therefore, we obtain
$$
G(K_2)=F_5 \rtimes \mathbb{Z},
$$
where $F_5=\langle y_0,\,y_1,\,y_2,\,y_3,\,y_4 \rangle$ is the free group of rank  5 and
$\mathbb{Z}=\langle x \rangle$.
\end{proof}
In the quotient  $F_5/ F_5 '$ the matrix
$$
[\alpha] =
\left(%
\begin{array}{ccccc}
  0 &  1 &  0 & 0 & 0\\
  0 &  0 &  1 & 0 & 0\\
  0 &  0 &  0 & 1 & 0\\
  0 &  0 &  0 & 0 & 1\\
  1 & -1 & -2 & 2 & 1\\
\end{array}%
\right)
$$
corresponds to conjugation by element $x$.
Since the characteristic polynomial of $A$ is 
$$
\chi(\lambda)=\lambda^5-\lambda^4-2\lambda^3+2\lambda^2+\lambda-1=(\lambda-1)^3(\lambda+1)^2,
$$
by second part of Theorem \ref{t}, we obtain the following result.

\begin{corollary}
The length of the lower central series of $G(K_2)$ is greater than or equal to $5$ and less than or equal to $\omega^2$.
\end{corollary}
\begin{proof}
Let us denote
$$
A=F_5/F_5'=\left\langle e_0, e_1, e_2, e_3, e_4\right\rangle \cong \mathbb{Z}^5,
$$
where $e_0, e_1, e_2, e_3, e_4$ are the images of  $y_0, y_1, y_2, y_3, y_4$, respectively in $A$.

It is easy to check that
$$
\mathrm{Ker} (\alpha -\id)^3= \left\langle v_1, v_2, v_3  \right\rangle, \quad
\mathrm{Ker} (\alpha +\id)^2= \left\langle w_1, w_2 \right\rangle,
$$
where
\begin{equation*}
\begin{array}{lcr}
v_1=e_0+2e_1+e_2, & v_2=-e_0-e_1+e_2+e_3, & v_3 =e_0-2e_2+e_4,\\
w_1=-e_0+3e_1-3e_2+e_3, & w_2=-e_0+2e_1-2e_3+e_4,\\
\end{array}
\end{equation*}
with
$$
\alpha : v_1 \mapsto v_1+v_2,\,\,  v_2 \mapsto v_2+v_3,\,\, v_3 \mapsto v_3,\,\,
w_1 \mapsto -w_1+w_2,\,\,  w_2 \mapsto -w_2.
$$
Let
$$
V= \left\langle v_1, v_2, v_3  \right\rangle, \quad W= \left\langle w_1, w_2 \right\rangle \text{ and } A_0=V+W.
$$
Then
$$
(\alpha -\id)^3 V=0 \text{ and } (\alpha -\id)^2 W\subseteq 2W.
$$
The quotient $A/A_0 \cong \mathbb{Z}_4 \oplus \mathbb{Z}_{16}$, where $\mathbb{Z}_4$ is generated by the image of  $t_0=e_0+3e_1$ and $\mathbb{Z}_{16}$  is generated by the image of  $t_1=e_1$
in $A/A_0$. Also, we have
$$
\alpha -\id : t_0 \mapsto 4t_0+4t_1,\,\,  t_1 \mapsto -t_0~  ({\rm mod} ~ A_0).
$$
Hence, $(\alpha -\id)^4 A \subseteq A_0$ and
$$
(\alpha -\id)^7 A \subseteq (\alpha -\id)^3 A_0 \subseteq 2 W \subseteq 2 A.
$$
The rest of the proof follows from the second part of Theorem \ref{t}.
\end{proof}

\begin{question}
Is it true that the length of the lower central series of $G(K_2)$ is equal to $\omega^2$?
\end{question}

\bigskip

\subsection{Virtual knot $K_3$ (The knot $3.7$ in \cite{table})}
$K_3$ is closure of a braid $\beta=\rho_1 \sigma_1^{-2} \rho_1 \sigma_1  \in VB_2$ as in Figure \ref{BraidForK_3} and we have,
$$G(K_3)=\langle\, x_1,\,x_2,\,y\, \| \, x_i=x_i\beta \, , i = 1, 2  \rangle, $$
where the relations are of  the form
$$
\left\{
\begin{array}{l}
  x_1=(x_1 x_2 x_1^{-1})^{-1} x_1^{-y^{-2}} (x_1 x_2 x_1^{-1}) x_1^{y^{-2}}  (x_1 x_2 x_1^{-1}) , \\
  x_2= (x_1 x_2 x_1^{-1})^{-y^{2}} x_1 (x_1 x_2 x_1^{-1})^{y^{2}}. \\
\end{array}
\right.$$

\begin{figure}[hbtp]
\centering
\includegraphics[scale=0.25]{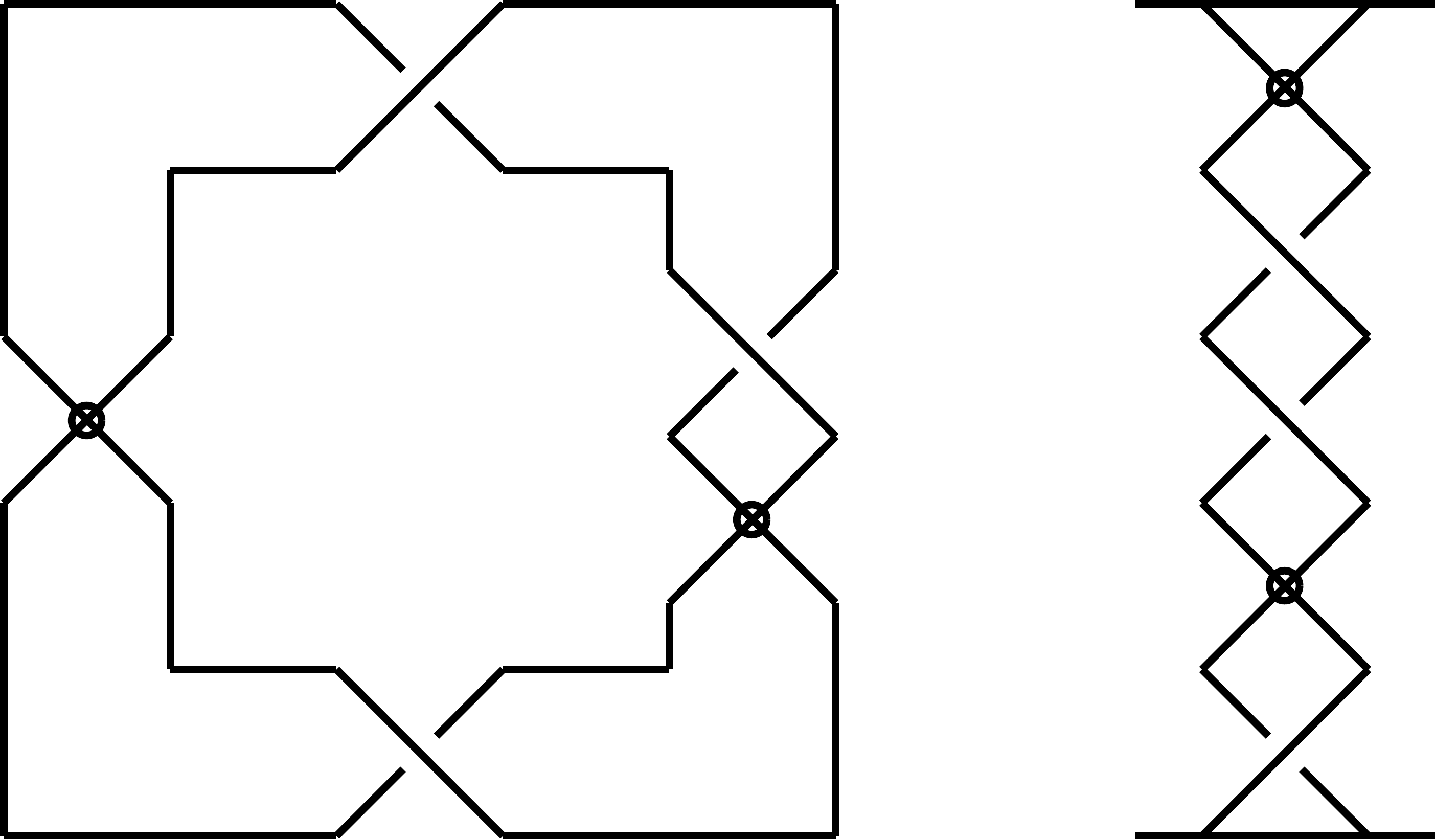}
\caption{Knot $K_3$ and its corresponding braid.}
\label{BraidForK_3}
\end{figure}

Transforming it, we obtain
$$
\left\{
\begin{array}{l}
  x_1=(x_1 x_2 x_1^{-1})^{-1} x_1^{-y^{-2}} (x_1 x_2 x_1^{-1}) x_1^{y^{-2}}  (x_1 x_2 x_1^{-1}) , \\
  x_2^{y^{-2}}= (x_1 x_2 x_1^{-1})^{-1}  x_1^{y^{-2}}  (x_1 x_2 x_1^{-1}), \\
\end{array}
\right.
\Longleftrightarrow
\left\{
\begin{array}{l}
  x_1=x_2^{-y^{-2}} x_1^{y^{-2}}  (x_1 x_2 x_1^{-1}) , \\
  x_2^{y^{-2}}= (x_1 x_2 x_1^{-1})^{-1}  x_1^{y^{-2}}  (x_1 x_2 x_1^{-1}), \\
\end{array}
\right.
$$
$$
\Longleftrightarrow \left\{
\begin{array}{l}
  x_1=x_2^{-y^{-2}} x_1^{y^{-2}}  (x_1 x_2 x_1^{-1}) , \\
  x_2^{y^{-2}}= (x_1 x_2 x_1^{-1})^{-1}  x_2^{y^{-2}}  x_1, \\
\end{array}
\right.
\Longleftrightarrow
\left\{
\begin{array}{l}
  x_2^{y^{-2}}= (x_1 x_2 x_1^{-1})^{-1}  x_2^{y^{-2}}  x_1 , \\
  x_1 x_2^{y^{-2}}=  x_2^{-y^{-2}}  x_1^{y^{-2}} x_2^{y^{-2}} x_1 , \\
\end{array}
\right.
$$
$$
\Longleftrightarrow \left\{
\begin{array}{l}
  x_1 x_2^{-1} x_1^{-1}  x_2^{y^{-2}}  x_1= x_2^{y^{-2}}, \\
  x_2^{-1} x_1 x_2 x_1^{y^{2}}=  x_1^{y^{2}} x_2, \\
\end{array}
\right.
\Longleftrightarrow
\left\{
\begin{array}{l}
  x_1 x_2^{-1}= x_2^{y^{-2}} x_1^{-1}  x_2^{-y^{-2}}  x_1 , \\
  x_2^{-1} x_1 = x_1^{y^{2}} x_2  x_1^{-y^{2}} x_2^{-1}. \\
\end{array}
\right.
$$

So, the group $G(K_3)$ has the presentation
$$
G(K_3) = \langle\, x_1,\, x_2,\,y\, \| \,
 x_1 x_2^{-1}= [x_2^{-y^{-2}}, x_1],\,\,
  x_2^{-1} x_1 = [x_1^{-y^{2}},  x_2^{-1}]\,  \rangle.
$$

\begin{theorem}
The following holds for the group $G(K_3)$:
\begin{enumerate}
\item The quotient $G(K_3)/\langle y \rangle^{G(K_3)}$ of $G(K_3)$ by the normal closure of $y$ in $G(K_3)$ is isomorphic to $$\langle\, x,\, z\, \| \,
 z^3=1,\,\,   z^x=z^{-1}\,  \rangle \ast \mathbb{Z}_2.
$$ In particular, it contains a subgroup that is isomorphic to the free product $\mathbb{Z}_3\ast \mathbb{Z}_2$.\\
\item
\begin{eqnarray*}
G(K_3) / \gamma_4 G(K_3) & \cong &  F_2/\gamma_4 F_2,\\
\gamma_4 G(K_3) /\gamma_5 G(K_3) & \cong &  \mathbb{Z}^2 \times \mathbb{Z}_4,\\
G(K_3) /\gamma_5 G(K_3) & \not\cong &  F_2/\gamma_5 F_2.
\end{eqnarray*}
\end{enumerate}
\end{theorem}

\begin{proof}
(1) We consider the second relation and we obtain
$$
  x_2^{-1} x_1 = \left[  x_1^{-1},  x_2^{-y^{-2}}   \right]^{y^{2}} \Longleftrightarrow  ( x_2^{-1} x_1)^{y^{-2}} = \left[  x_1^{-1},  x_2^{-y^{-2}}   \right]
$$
$$
\Longleftrightarrow ( x_2^{-1} x_1)^{y^{-2}} = \left[  x_2^{-y^{-2}}, x_1  \right]^{x_1^{-1}} \Longleftrightarrow  \left[  x_2^{-y^{-2}}, x_1^{-1}  \right]= ( x_2^{-1} x_1)^{y^{-2}x_1}.
$$
Using the first relation
$$
 ( x_2^{-1} x_1)^{y^{-2}x_1}=x_1 x_2^{-1}=( x_2^{-1} x_1)^{x_1^{-1}},
$$
we get
$$
 [ x_2^{-1} x_1, y^{-2}x_1^2]=1.
$$
Denote $x_1=x$, $y=y$, $x_2^{-1} x_1=z$, then  $x_2^{-1} =z x_1^{-1}$
and the relations are of the form
$$
 [ z, y^{-2}x^2]=1,\quad  xzx^{-1}=[(z x^{-1})^{y^{-2}},x].
$$
In the quotient  $G(K_3)/\langle y^2 \rangle^{G(K_3)}$, relations are of the form
$$
 [ z, x^2]\equiv 1,\quad  xzx^{-1}\equiv [z x^{-1},x].
$$
The second relation is equivalent to relation  $z^x=z^2$ and from the first relation it follows that
$$z^3=1.$$
So,
$$
G(K_3) / \langle y^2 \rangle^{G(K_3)} \cong \langle\, x,\, z\, \| \,
 z^3=1,\,\,   z^x=z^{-1}\,  \rangle.
$$
Therefore, we obtain
$$
G(K_3)/\langle y \rangle^{G(K_3)} \cong \langle\, x,\, z\, \| \,
 z^3=1,\,\,   z^x=z^{-1}\,  \rangle \ast \mathbb{Z}_2.
$$

\vspace{0.5cm}

(2)  Considering the relation
$ x_2^{-1} x_1 = [x_1^{-y^{2}},  x_2^{-1}]$
and transforming it, we get
$$
  x_1 = x_2 [x_1^{-y^{2}},  x_2^{-1}] = x_2 [y^{-2}x_1^{-1}y^{2},  x_2^{-1}] = x_2 [ [y^2,x_1] x_1^{-1},  x_2^{-1}]=
$$
$$
  = x_2 [ y^2,  x_1,  x_2^{-1}]^{x_1^{-1}}  [ x_1^{-1},  x_2^{-1}] = x_2 [ y^2,x_1,x_2^{-1}] [ y^2,  x_1,  x_2^{-1},x_1^{-1}] [ x_1^{-1},  x_2^{-1}]=
$$
$$
  = x_2 [ y^2,x_1,x_2^{-1}] [ y^2,  x_1,  x_2^{-1},x_1^{-1}] [ x_2^{-1},x_1 ]^{x_1^{-1}}=
$$
$$
  = x_2 [ y^2,x_1,x_2^{-1}] [ y^2, x_1, x_2^{-1},x_1^{-1}] [ x_2^{-1},x_1 ] [ x_2^{-1},x_1, x_1^{-1} ],
$$
which implies
$$
  x_1= x_2 [ y^2,x_1,x_2^{-1}] [ y^2, x_1, x_2^{-1},x_1^{-1}] [ x_2^{-1},x_1 ] [ x_2^{-1},x_1, x_1^{-1} ].
$$
Performing the transformations modulo $\gamma_4 G(K_3)$,  we get
$$
  x_1 \equiv x_2 [ y,x_2,x_2]^{-2} [ x_2^{-1},x_1 ] [ x_2,x_1, x_1 ] \equiv x_2 [ y,x_2,x_2]^{-2} [ x_2^{-1},x_1 ] \equiv
$$
$$
  \equiv x_2 [ y,x_2,x_2]^{-2} [ x_2^{-1},  x_2 [ x_2^{-1},x_1 ] ] \equiv  x_2 [ y,x_2,x_2]^{-2} [ x_2^{-1},  [ x_2^{-1},x_1 ] ] \equiv x_2 [ y,x_2,x_2]^{-2} .
$$
Hence, the relation
$ x_2^{-1} x_1 = [x_1^{-y^{2}},  x_2^{-1}]$
has the form
$$
x_1 \equiv x_2 [ y,x_2,x_2]^{-2} \, (\mathrm{mod}~ \gamma_4 G(K_3)).
$$
Further, in the quotient $G(K_3) / \gamma_4 G(K_3)$, the second relation
$ x_1 = [x_2^{-y^{-2}}, x_1]x_2$
has the form
\begin{eqnarray*}
  x_2 [ y,x_2,x_2]^{-2} & \equiv& [x_2^{-y^{-2}}, x_2]x_2,\\
  x_2 [ y,x_2,x_2]^{-2} & \equiv& [[y^{-2},x_2] x_2^{-1}, x_2]x_2,\\
  x_2 [ y,x_2,x_2]^{-2} & \equiv& [y^{-2},x_2 , x_2]^{-2}x_2 \equiv 1.
\end{eqnarray*}
So, we have
$$
G(K_3) / \gamma_4 G(K_3) \cong \langle\, x_1,\,x_2,\, y\, \| \,
 x_1 = x_2 [ y,x_2,x_2]^{-2}\,  \rangle \cong  F_2/\gamma_4 F_2.
$$
Let us now consider the following relation
$$
 [ x_2^{-1} x_1, y^{-2}x_1^2] = 1,
$$
then transforming it modulo  $\gamma_5 G(K_3)$ and using the relation
$$
  x_1 \equiv x_2 [ y,x_2,x_2]^{-2} \, (\mathrm{mod} ~\gamma_4 G(K_3)),
$$
we get
\begin{eqnarray*}
 \lbrack x_2^{-1} x_2 [ y,x_2,x_2]^{-2}, y^{-2} x_2^2 [ y,x_2,x_2]^{-4} \rbrack &\equiv&1 (\mathrm{mod}~ \gamma_5 G(K_3)),\\
  \lbrack y,x_2,x_2, y \rbrack^4 [ y,x_2,x_2,x_2]^{-4}&\equiv& 1 (\mathrm{mod}~ \gamma_5 G(K_3)),\\
 ([ y,x_2,x_2, y] [ y,x_2,x_2,x_2]^{-1})^{4} &\equiv& 1 (\mathrm{mod}~ \gamma_5 G(K_3)).
\end{eqnarray*}
Hence, we finally have
\begin{eqnarray*}
G(K_3) / \gamma_4 G(K_3) & \cong &  F_2/\gamma_4 F_2,\\
\gamma_4 G(K_3) /\gamma_5 G(K_3) & \cong &  \mathbb{Z}^2 \times \mathbb{Z}_4,\\
G(K_3) /\gamma_5 G(K_3) & \not\cong &  F_2/\gamma_5 F_2.
\end{eqnarray*}
\end{proof}

\subsection{Virtual knot $K_4$ (The knot $4.43$ in \cite{table})}

We compute $G(K_4)$ using its knot diagram. By labelling the arcs as shown in Figure \ref{BraidForK_4}, we get
$$G(K_4)= \langle x_1, x_2, x_3, x_4, x_5, x_6, y ~|~ x_1x_5=x_6x_1, x_5x_1=x_2x_5, x_2x_4=x_5x_2, x_6x_2=x_3x_6, x_4 = x_3^y,  x_6 = x_1^y \rangle.$$

Eliminating $x_4$ and $x_6$, we get 
$$G(K_4)= \langle x_1, x_2, x_3, x_5, y ~|~ x_1x_5=x_1^yx_1, x_5x_1=x_2x_5, x_2x_3^y=x_5x_2, x_1^yx_2=x_3x_1^y \rangle.$$

From the defining relations, we get 
$$x_5=x_1^{yx_1}.$$

We eliminate $x_5$ and we have
$$G(K_4)= \langle x_1, x_2, x_3, y ~|~  x_1^{yx_1}x_1=x_2x_1^{yx_1}, x_2x_3^y=x_1^{yx_1}x_2, x_1^yx_2=x_3x_1^y \rangle.$$

Next, we have
$$x_2=x_1^{yx_1}x_1x_1^{-yx_1}.$$

So we eliminate $x_2$ and we get
$$G(K_4)= \langle x_1, x_3, y ~|~  x_1^{yx_1}x_1x_1^{-yx_1}x_3^y=x_1^{yx_1}x_1^{yx_1}x_1x_1^{-yx_1}, x_1^yx_1^{yx_1}x_1x_1^{-yx_1}=x_3x_1^y \rangle.$$

Writing $x_3$ in terms of $x_1$ and $y$ and replacing $x_1$ with $x$, we get
$$G(K_4)= \langle x, y ~|~  x^{yx}xx^{-yx}(x^yx^{yx}xx^{-yx}x^{-y})^y=x^{yx}x^{yx}xx^{-yx} \rangle.$$

Simplifying the relation, $G(K_4)$ has the presentation
$$
G(K_4) = \langle\, x,\,y\, \| \, x^{y^{-1} x^{-1} y x y^{-1} x^{-1} y^2}=x^{y x y^{-1}x^{-1} y x} \,  \rangle.
$$

\begin{figure}[hbtp]
\centering
\includegraphics[scale=0.7]{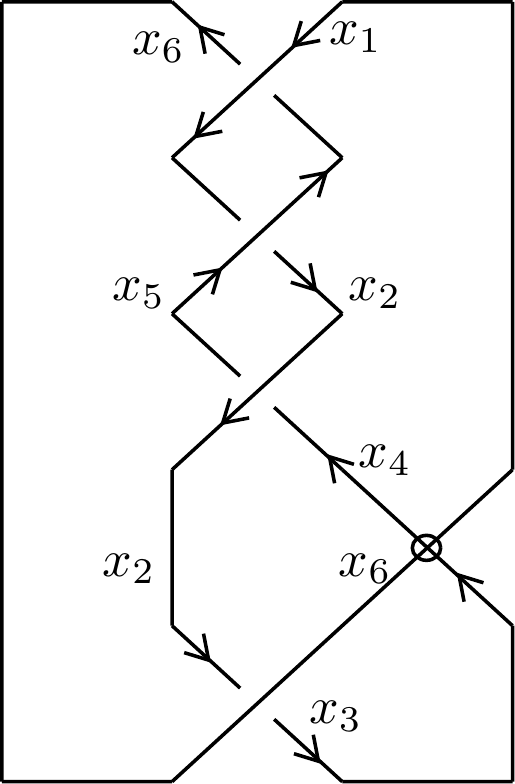}
\caption{Knot $K_4$.}
\label{BraidForK_4}
\end{figure}

\begin{theorem} The following holds for $G(K_4)$:
\begin{enumerate}
\item The group $G(K_4)$ is isomorphic to $G(K_4) = H_4\rtimes \mathbb{Z}$, where $$
H_4=\cdots \underset{B_{k-1}}{\ast} A_{k}\underset{B_{k}}{\ast} A_{k+1}
  \underset{B_{k+1}}{\ast} A_{k+2} \underset{B_{k+2}}{\ast} \cdots
$$
is the amalgamated free product with $$
A_k=\langle\, x_k,\, x_{k+1},\, x_{k+2} \, \| \,
    [x_{k+1}, x_{k+2}^{-1}  x_{k} x_{k+2}^{-1} x_{k}^{-1} x_{k+1} x_{k}^{-1}]=1 \,  \rangle,
$$ and $B_k$ is free group of rank 2, $k \in \mathbb{Z}$.

\item The quotient of $G(K_4)$ by the terms of lower central series gives
\begin{eqnarray*}
G(K_4) / \gamma_4 G(K_4) & \cong &  F_2/\gamma_4 F_2,\\
\gamma_4 G(K_4) / \gamma_5 G(K_4) & \cong &  \mathbb{Z}^2 \times \mathbb{Z}_2,\\
G(K_4) / \gamma_5 G(K_4) & \not\cong &  F_2/\gamma_5 F_2.
\end{eqnarray*}
\end{enumerate}
\end{theorem}

\begin{proof}~~
 \enumerate
\item
We consider the relation
$$
x^{y^{-1} x^{-1} y x y^{-1} x^{-1} y^2}=x^{y x y^{-1}x^{-1} y x}.
$$
Transforming it, we obtain
$$
x^{y^{-1} x^{-1} y x y^{-1} x^{-1} y^2 x^{-1} y^{-1}x y x^{-1} y^{-1}}=x,$$
which implies
$$ [x, y^{-1} x^{-1} y x y^{-1} x^{-1} y^2 x^{-1} y^{-1}x y x^{-1} y^{-1}]=1.
$$
If we introduce the notation
$$
x_k=y^{-k} x y^{k},\quad k \in \mathbb{Z},
$$
then we can rewrite the above relation as
$$
[x_0, x_1^{-1}  x_0 x_1^{-1} x_{-1}^{-1} x_0 x_{-1}^{-1}]=1.
$$
Let
$$
H_4=\langle\, x_k,\,k\in \mathbb{Z}\, \| \,
    [x_{k+1}, x_{k+2}^{-1}  x_{k} x_{k+2}^{-1} x_{k}^{-1} x_{k+1} x_{k}^{-1}]=1,\,k\in \mathbb{Z} \,  \rangle.
$$
Then
$$
G(K_4) = H_4\rtimes \mathbb{Z},
$$
where $\mathbb{Z}= \langle y \rangle$ and $x_k^{y}=x_{k+1}$, $k\in \mathbb{Z}$.
Further, for every   $k\in \mathbb{Z}$ let
$$
A_k=\langle\, x_k,\, x_{k+1},\, x_{k+2} \, \| \,
    [x_{k+1}, x_{k+2}^{-1}  x_{k} x_{k+2}^{-1} x_{k}^{-1} x_{k+1} x_{k}^{-1}]=1 \,  \rangle,
$$
and
$$
B_k=\langle\,  x_{k+1},\, x_{k+2}\, \rangle.
$$
Then
$$
H_4=\cdots \underset{B_{k-1}}{\ast} A_{k}\underset{B_{k}}{\ast} A_{k+1}
  \underset{B_{k+1}}{\ast} A_{k+2} \underset{B_{k+2}}{\ast} \cdots
$$
is an infinite  free product with amalgamation.
Note that  $B_k\cong F_2$ is a free group of rank  2,  $k\in \mathbb{Z}$.
\bigskip

\item
Transforming the relation, we get
\begin{eqnarray*}
x^{y^{-1} x^{-1} y x y^{-1} x^{-1} y^2} &=& x^{y x y^{-1}x^{-1} y x},\\
x^{y^{-1} x^{-1} y x y^{-1} x^{-1} y x} &=&  x^{y x y^{-1}x^{-1} y x y^{-1}x},\\
x^{[y, x]^2 } &=& x^{x^{-2}y x [y,x] y^{-1}x},\\
\lbrack[y, x]^2,x \rbrack &=&  [x^{-2}y x [y,x] y^{-1}x, x],\\
\lbrack y, x,x \rbrack ^{[y, x]} [y, x,x] &=&  [x^{-2}y x y^{-1}x [y,x] [y,x,y^{-1}x], x],\\
\lbrack y,x,x \rbrack [y, x,x,[y, x]] [y,x,x] &=&  [[x,y^{-1}]^{x} [y,x] [y,x,y^{-1}x], x].
\end{eqnarray*}

Performing the transformations modulo  $\gamma_5 G(K_4)$, we obtain
\begin{eqnarray*}
\lbrack y,x,x \rbrack^2 &\equiv & [[x,y^{-1}] [x,y^{-1},x] [y,x] [y,x,y^{-1}x], x],\\
\lbrack y,x,x \rbrack^2 &\equiv & [[x,y^{-1}] [y,x], x] [x,y,x,x]^{-1}  [y,x,y,x]^{-1} [y,x,x,x],\\
\lbrack y,x,x \rbrack^2 &\equiv & [[y,x]^{y^{-1}} [y,x], x] [y,x,x,x]^{2}  [y,x,y,x]^{-1},\\
\lbrack y,x,x \rbrack^2 &\equiv & [[y,x] [y,x,y^{-1}] [y,x], x] [y,x,x,x]^{2}  [y,x,y,x]^{-1},\\
\lbrack y,x,x \rbrack^2 &\equiv & [[y,x]^2 , x] [y,x,x,x]^{2}  [y,x,y,x]^{-2},\\
1&\equiv &  [y,x,x,x]^{2}  [y,x,y,x]^{-2}.
\end{eqnarray*}

Therefore, in the quotient $G(K_4)/\gamma_5 G(K_4)$, the relation is of the form
$$
\left( [x,y,y,x]  [x,y,x,x]^{-1}\right)^2 \equiv 1.
$$
Since the commutators $[x,y,y,x]$, $[x,y,x,x]$, $[x,y,y,y]$
are the basic commutators of weight  4, we have
\begin{eqnarray*}
G(K_4) / \gamma_4 G(K_4) & \cong &  F_2/\gamma_4 F_2,\\
\gamma_4 G(K_4) / \gamma_5 G(K_4) & \cong &  \mathbb{Z}^2 \times \mathbb{Z}_2,\\
G(K_4) / \gamma_5 G(K_4) & \not\cong &  F_2/\gamma_5 F_2.
\end{eqnarray*}
\end{proof}

\section{Groups $G_A$ and $G_{\tilde{M}}$ for  the virtual Hopf link} \label{section4}
\begin{proposition}
The group $G_A(L)$ and $G_{\tilde{M}}(L)$ of virtual Hopf link $L$ are right angled Artin groups and $G_A(L)$ is not isomorphic to $G_{\tilde{M}}(L)$. 
\end{proposition}

\begin{proof}
 It is easy to see that $L$ is equivalent to closure of braid $\beta=\sigma_1^{-1}\rho_1 \in VB_2$ as depicted in Figure \ref{BraidForHopfLink}.
We have
$$G_A(\beta)=\langle\, x_1,\,x_2,\,y\, \| \, [y, x_1]\,=\,[x_1, x_2] =1 \,  \rangle $$
which is isomorphic to the direct product $\mathbb{Z} \times F_2$, where $\mathbb{Z} = \langle x_1 \rangle$ and $F_2 = \langle x_2, y \rangle$.
On the other hand, we have
$$G_{\tilde{M}}(K)=\langle\, y_1, \, y_2,\, v_1 ,\,v_2 \|\, [y_1,v_2]=[y_1,y_2]=[v_1,v_2]=1 \, \rangle$$
which is isomorphic to the quotient of the free product $\mathbb{Z}^2 * \mathbb{Z}^2$ by the normal closure of the commutator $[y_1, v_2]$, where the first factor  $\mathbb{Z}^2 = \langle y_1, y_2 \rangle$ and the second factor $\mathbb{Z}^2 = \langle v_1, v_2 \rangle$.\\
\\
Hence, we have
$$G_A(\beta)/\gamma_2 G_A(\beta) \cong \mathbb{Z}^3,$$
and
$$G_{\tilde{M}}(\beta)/\gamma_2 G_{\tilde{M}}(\beta) \cong \mathbb{Z}^4.$$
which implies
$$
G_A(L) \not \cong G_{\tilde{M}}(L).
$$
\end{proof}

\begin{figure}[hbtp]
\centering
\includegraphics[scale=0.35]{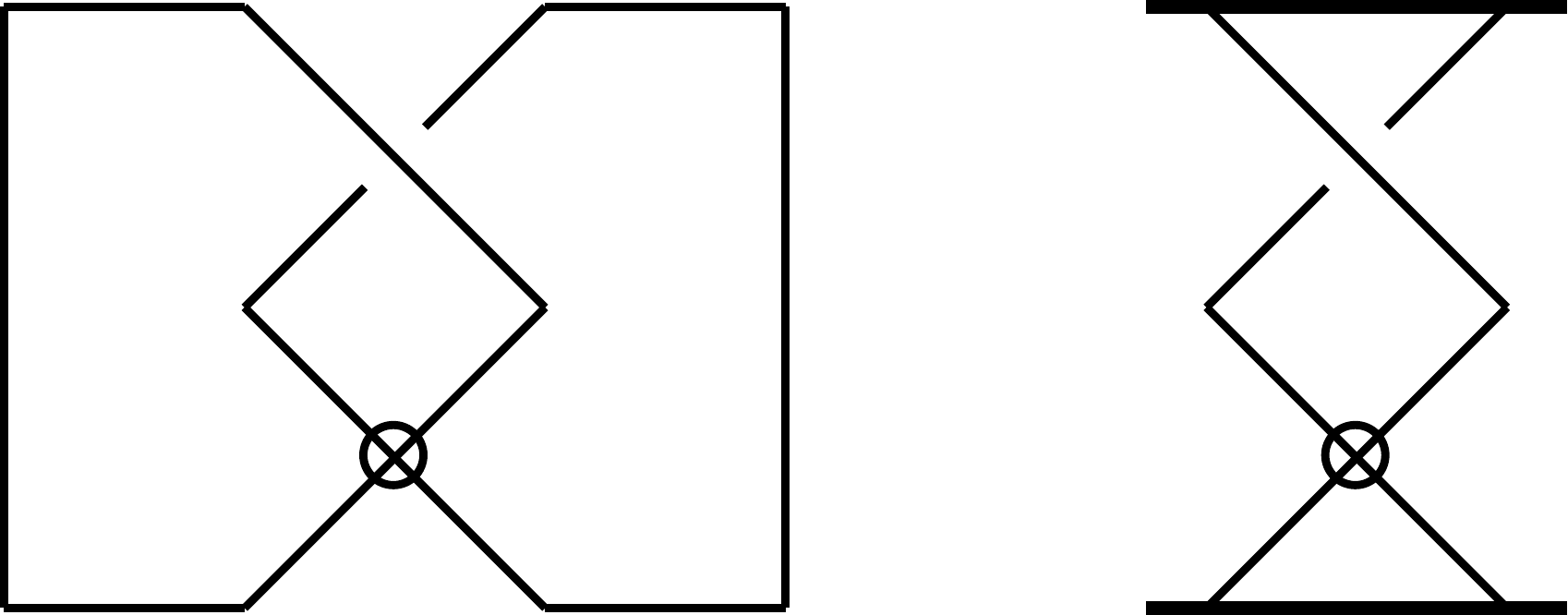}
\caption{Hopf Link and its corresponding braid}
\label{BraidForHopfLink}
\end{figure}

In \cite{Mikhailov}, a one-relator group is constructed whose lower central series has length equal to $\omega^2$.
\begin{problem}
Is there a virtual knot $K$ such that the length of the lower central series of 
$G_A(K)$ is more than $\omega$? 
\end{problem}

\begin{problem}
We conclude with the following problems:
\begin{enumerate}
\item Suppose that $K$ is a non-classical virtual knot. Is it true that
$$
\gamma_2 G_A(K) \not= \gamma_3 G_A(K)?
$$
In other words, is it true that the lower central series distinguish classical knots? 

\item Which groups of virtual knots are residually nilpotent?

\end{enumerate}
\end{problem}


\noindent\textbf{Acknowledgments.}
This work was supported by  the Russian Science Foundation (project No. 19-41-02005). Neha Nanda would like to thank her advisor Mahender Singh for his
comments and corrections in an earlier version of the paper. She also
thanks IISER Mohali for the PhD Research Fellowship and the DST grant
INT/RUS/RSF/P-02. The authors would also like to thank the referee for valuable suggestions and corrections.

\end{document}